\documentclass[11pt,leqno]{article}

\usepackage{mathrsfs,amssymb,amsmath,amsthm,amsfonts, color}
\usepackage[dvips]{graphicx}

\makeatletter
	
	\@addtoreset{equation}{section}
\makeatother

\begin{document}

\renewcommand{\theenumi}{\rm (\roman{enumi})}
\renewcommand{\labelenumi}{\rm \theenumi}

\newtheorem{thm}{Theorem}[section]
\newtheorem{defi}[thm]{Definition}
\newtheorem{lem}[thm]{Lemma}
\newtheorem{prop}[thm]{Proposition}
\newtheorem{cor}[thm]{Corollary}
\newtheorem{exam}[thm]{Example}
\newtheorem{conj}[thm]{Conjecture}
\newtheorem{rem}[thm]{Remark}
\allowdisplaybreaks

\title{An improvement of the integrability of the state space of the $\Phi^4_3$-process and the support of the $\Phi ^4_3$-measure constructed by the limit of stationary processes of approximating stochastic quantization equations}

\author{Seiichiro Kusuoka
\vspace{5mm}\\
\normalsize Department of Mathematics, Graduate School of Science, Kyoto University,\\
\normalsize Kitashirakawa Oiwakecho, Sakyo-ku, Kyoto 606-8502, Japan\\
\normalsize e-mail address: {\tt{kusuoka@math.kyoto-u.ac.jp}}}
\maketitle

\begin{abstract}
This is a remark paper for the $\Phi ^4_3$-measure and the associated flow on the torus which are constructed in \cite{AlKu} by the limit of the stationary processes of the stochastic quantization equations of approximation measures.
We improve the integrability of the state space of the $\Phi ^4_3$-process and the support of the $\Phi ^4_3$-measure.
For the improvement, we improve the estimates of the H\"older continuity in time of the solutions to approximation equations.
In the present paper, we only discuss the estimates different from those in \cite{AlKu}.
\end{abstract}

{\bf AMS Classification Numbers:}
60H17, 81S20, 81T08, 60H15, 35Q40, 35R60, 35K58

 \vskip0.2cm

\renewcommand{\labelenumi}{\rm{(\roman{enumi}})}

\section{Introduction}\label{sec:intro}

Recently by the new theories such as regularity structure \cite{Ha} and paracontrolled calculus \cite{GIP}, singular nonlinear stochastic partial differential equations became solvable via renormalization.
In particular, the singular stochastic partial differential equations associated to stochastic quantization of the $\Phi ^4_3$-measure are solved (see \cite{AlKu}, \cite{CaCh}, \cite{GuHo1}, \cite{GuHo2}, \cite{Ho}, \cite{HIN}, \cite{MoWe}, \cite{MW3}, \cite{ZhZh1} and \cite{ZhZh2}).
Moreover, the $\Phi ^4_3$-measure is able to be constructed from the stochastic quantization equations (see \cite{AlKu}, \cite{BaGu1}, \cite{BaGu2}, \cite{GuHo2} and \cite{MW3}).
For the detail of the history and background of the $\Phi ^4_3$-measure and stochastic quantization, see the introduction of \cite{AlKu}.

In \cite{AlKu}, we considered the probability measures which approximate the $\Phi ^4_3$-measure, and the stochastic quantization equations associated to them, and provided the stationary solutions to the approximating stochastic quantization equations. By proving the tightness of the stationary solutions we obtain the $\Phi ^4_3$-process as a limit.
Moreover, we constructed the $\Phi ^4_3$-measure as a limit of the marginal distributions. Here, note that the approximation sequence of the marginal distributions is an approximation of the (formally defined) $\Phi ^4_3$-measure.
The most remarkable advantage of considering the stationary solutions is that we are able to construct the time-global limit process and the $\Phi ^4_3$-measure directly. This is a difference between \cite{AlKu} and the earlier result \cite{MW3}.
We remark that there is another delicate difference between \cite{AlKu} and \cite{MW3}.
In \cite{AlKu}, we first prepare the probability measures $\{ \mu _N\}$ approximating the (formally defined) $\Phi ^4_3$-measure, and consider the stochastic partial differential equations associated to the stochastic quantization of $\{ \mu _N\}$.
On the other hand, in \cite{MW3}, they first consider the stochastic quantization equation associated to the  (formally defined) $\Phi ^4_3$-measure and show the existence of the global solution to the stochastic quantization equation by approximation.
So, between the arguments of \cite{AlKu} and \cite{MW3} there is a difference on the order of the two operations: approximation and stochastic quantization.
This makes a delicate difference in the concerned stochastic partial differential equations.
Indeed, approximation operators appear in the stochastic quantization equation in the case of \cite{AlKu} (see Eq. (4.1) in \cite{AlKu}).
Because of the difference, we only have an energy functional with square and fourth-power integrals in \cite{AlKu}, while the $p$th-power integrability of energy functionals is obtained for all $p\in [1,\infty )$ in \cite{MW3}.
Hence, we have some restriction on the integrability of the function spaces in the argument of \cite{AlKu}.

In the present paper, we improve the integrability of the state space of the $\Phi ^4_3$-process and the support of the $\Phi ^4_3$-measure obtained by \cite{AlKu}.
We will show the tightness of the approximating processes in smaller Besov spaces by improving the estimates of the H\"older continuity in time (see Proposition \ref{prop:holder+}) and the estimate uniform in time (see Proposition \ref{prop:estsup+}). 
They enable us to improve the main estimate in \cite{AlKu} (see Theorem \ref{thm:tight1}) and by using the estimate and the Besov embedding theorem we obtain the better integrability of the state space $B_{12/5}^{-1/2-\varepsilon}$ for the limit process and the support $B_\infty^{-1/2-\varepsilon}$ of our $\Phi ^4_3$-measure (see Theorem \ref{thm:tight2}).
We remark that in the setting of \cite{MW3}, which is different from our setting as mentioned above, much more integrability for the state space of the $\Phi ^4_3$-process is obtained.
On the other hand, the supports of  the $\Phi ^4_3$-measures obtained here and obtained in \cite{MW3} are the same.

We also remark that the state space of the $\Phi ^4_3$-process and the support of the $\Phi ^4_3$-measure obtained in the present paper are different.
Note that null sets of the $\Phi ^4_3$-measure can be ignored in the support of the measure, but cannot in the state space of the $\Phi ^4_3$-process.
Only polar sets can be ignored in the state space of the $\Phi ^4_3$-process.
Moreover, generally polar sets of processes are smaller than null sets of the invariant measures.
Hence, such a difference naturally appears in the main theorem (see Theorem \ref{thm:tight2}).


The organization of the present paper is as follows.
In Section \ref{sec:preparation} we recall the notation and setting of \cite{AlKu}.
In Section \ref{sec:improve} we consider the improvement of the integrability.
To do it, we give some estimates better than those in \cite{AlKu}.
We only discuss the different parts of the argument in \cite{AlKu} and show the main theorem (Theorem \ref{thm:tight2}).

\section{Preparation}\label{sec:preparation}

In this section we recall the notation and setting of \cite{AlKu}.
Let $\Lambda$ be the three-dimensional torus given by $({\mathbb R}/(2\pi {\mathbb Z}))^3$.
Let $L^p$ and $W^{s,p}$ be the $p$th-order integrable function space and the Sobolev space respectively, with respect to the Lebesgue measure on $\Lambda$, for $s\in {\mathbb R}$ and $p\in [1,\infty ]$.
Denote by $\langle \cdot ,\cdot \rangle$ the inner product on $L^2(\Lambda; {\mathbb C})$.
Let $\{ e_k; k\in {\mathbb Z}^3\}$ be the Fourier basis on $L^2(\Lambda; {\mathbb C})$ and $k^2 := \sum _{j=1}^3 k_j^2$ for $k = (k_1,k_2,k_3) \in {\mathbb Z}^3$.

To define approximation operators on ${\mathcal D}'(\Lambda)$ (the space of distributions on $\Lambda$), let $\psi ^{(1)}$ be a nonincreasing $C^\infty $-function on $[0,\infty)$ such that $\psi ^{(1)}(r) =1$ for $r\in [0,1]$ and $\psi ^{(1)}(r)=0$ for $r\in [2,\infty )$, and let $\psi ^{(2)}$ be a nonincreasing function on $[0,\infty)$ such that $\psi ^{(2)}(r) =1$ for $r\in [0,2]$ and $\psi ^{(2)}(r)=0$ for $r\in [4,\infty )$.
We remark that $\psi ^{(2)}$ is not necessary continuous.
For $N\in {\mathbb N}$, $i=1,2$ and $k = (k_1,k_2,k_3) \in {\mathbb Z}^3$, denote $\psi ^{(i)}(2^{-N} |k_1|)\psi ^{(i)}(2^{-N} |k_2|)\psi ^{(i)}(2^{-N} |k_3|)$ by $\psi _N^{(i), \otimes 3}(k)$, and define $P_N^{(i)}$ by the mapping from ${\mathcal D}'(\Lambda)$ to $C^\infty (\Lambda )$  given by
\[
P_N ^{(i)} f := \sum _{k\in {\mathbb Z}^3} \psi _N^{(i), \otimes 3}(k) \langle f , e_k \rangle e_k .
\]

Let $\mu _0$ be the centered Gaussian measure on ${\cal D}'(\Lambda )$ with the covariance operator $[2(-\triangle + m_0^2)]^{-1}$ where $\triangle$ is the Laplacian on $\Lambda$ and $m_0>0$, and let
\begin{align*}
C_1^{(N)} &:= \frac{1}{2(2\pi )^3} \sum _{k\in {\mathbb Z}^3} \frac{\left( \psi _N^{(1), \otimes 3}(k) \right)^2}{k^2+m_0^2} \\
C_2^{(N)} &:= \frac{1}{2(2\pi )^6} \sum _{l_1,l_2 \in {\mathbb Z}^3} \frac{\left( \psi _N^{(1), \otimes 3}(l_1) \right) ^2 \left( \psi _N^{(1), \otimes 3}(l_2) \right) ^2 \left( \psi _N^{(1), \otimes 3}(l_1+l_2) \right) ^2}{(l_1^2+m_0^2)(l_2^2+m_0^2)(l_1^2 + l_2^2 + (l_1+l_2)^2 +3m_0^2)} .
\end{align*}
The constants $C_1^{(N)}$ and $C_2^{(N)}$ are renormalization constants, and satisfy $\lim _{N\rightarrow \infty} C_1^{(N)} = \lim _{N\rightarrow \infty} C_2^{(N)} =\infty$.
Let $\lambda _0 \in (0,\infty )$ and $\lambda \in (0,\lambda _0]$ be fixed.
Define a function $U_N$ on ${\cal D}'(\Lambda)$ by
\[
U_N(\phi ) = \int _{\Lambda} \left\{ \frac{\lambda}{4} (P_N^{(1)} \phi ) (x)^4 - \frac{3\lambda }{2}\left( C_1^{(N)} -3\lambda C_2^{(N)}\right) (P_N ^{(1)} \phi )(x)^2 \right\} dx ,
\]
and consider the probability measure $\mu _N$ on ${\cal D}'(\Lambda)$ given by
\[
\mu _N (d\phi ) = A_N ^{-1} \exp \left( -U_N(\phi ) \right) \mu _0 (d\phi)
\]
where $A_N$ is the normalizing constant.
We remark that $\{ \mu _N\}$ is an approximation sequence for the $\Phi ^4_3$-measure which will be constructed below as a stationary probability measure of the flow associated with the stochastic quantization equation.

Letting $\dot W_t(x)$ be a Gaussian white noise with parameter $(t,x)\in  (-\infty ,\infty) \times \Lambda$, we consider the stochastic partial differential equation on $\Lambda$
\begin{equation}\label{eq:SDEN3}\left\{ \begin{array}{rl}
\displaystyle \partial _t \tilde X_t^{N}(x)
&\displaystyle = \dot W_t(x) - (-\triangle +m_0^2) \tilde X_t^N(x) \\[3mm]
&\displaystyle \quad - \lambda P_N ^{(1)} \left\{ (P_N ^{(1)} \tilde X_t^N)^3 (x) -3 \left( C_1^{(N)} -3\lambda C_2^{(N)}\right) P_N ^{(1)} \tilde X_t^N(x) \right\} \\
\displaystyle \tilde X_0^{N}(x)
&\displaystyle = \xi _N (x)
\end{array}\right. \end{equation}
where $\xi _N$ is an initial value which has $\mu _N$ as its law and is independent of $\dot W_t$.
Then, $\tilde X^{N}$ is a stationary process (see Theorem 4.1 of \cite{AlKu}).
Supplementary we prepare $Z_t$ defined by the solution to the stochastic partial differential equation on  $\Lambda$:
\begin{equation}\label{eq:SDEZ}\left\{ \begin{array}{rll}
\partial _t Z_t(x) &= \dot W_t(x) - (-\triangle +m_0^2)Z_t(x) , & (t,x)\in (-\infty ,\infty) \times \Lambda \\
Z_0(x)& =\zeta (x), &x\in \Lambda
\end{array}\right.\end{equation}
where $\zeta$ is a random variable which has $\mu _0$ as its law and is independent of $\dot W_t$.
We choose a pair of the initial values $(\xi _N ,\zeta)$ so that the paired process $(\tilde X^{N}, Z)$ is a stationary process.
For the existence of such a pair, see Section 4 of \cite{AlKu}.

Next we prepare notation of Besov spaces and paraproducts.
Let $\chi$ and $\varphi$ be functions in $C^\infty ([0,\infty );[0,1])$ such that the supports of $\chi$ and $\varphi$ are included by $[0,4/3)$ and $[3/4, 8/3]$ respectively, and that
\[
\chi (r ) + \sum _{j=0}^\infty \varphi (2^{-j}r ) =1, \quad r\in [0,\infty ).
\]
Then, it is easy to see that
\begin{align*}
&\varphi (2^{-j}r )  \varphi (2^{-k}r ) =0, \quad r \in [0,\infty ),\ j,k \in {\mathbb N}\cup \{ 0\} \ \mbox{such that}\ |j-k| \geq 2,\\
&\chi (r ) \varphi (2^{-j}r ) =0, \quad r \in [0,\infty),\ j \in {\mathbb N}.
\end{align*}
Let ${\mathcal S}({\mathbb R}^3)$ and ${\mathcal S}'({\mathbb R}^3)$ be the Schwartz space and the space of tempered distributions on ${\mathbb R}^3$, respectively.
For $f\in {\mathcal D}'(\Lambda )$, we can define the periodic extension $\widetilde{f} \in {\mathcal S}'({\mathbb R}^3)$.
By this extension, we define the (Littlewood-Paley) nonhomogeneous dyadic blocks $\{ \Delta _j; j\in {\mathbb N} \cup \{ -1, 0\}\}$ by setting
\[
\begin{array}{lll}
\Delta _{-1} f (x)&= \left[ {\mathcal F}^{-1} \left( \chi (|\cdot |) {\mathcal F} \widetilde{f} \right) \right] (x), \quad & x\in \Lambda \\
\Delta _j f (x)&= \left[ {\mathcal F}^{-1} \left( \varphi (2^{-j} |\cdot |){\mathcal F} \widetilde{f} \right) \right] (x), \quad & x\in \Lambda, \ j \in {\mathbb N}\cup \{ 0\} ,
\end{array}
\]
where ${\mathcal F}$ and ${\mathcal F}^{-1}$ are the Fourier transform and inverse Fourier transform operators on ${\mathbb R}^3$.
We remark that 
\[
\Delta _{-1} f = \sum _{k\in {\mathbb Z}^3} \chi (|k|) \langle f, e_k \rangle e_k , \quad
\Delta _j f = \sum _{k\in {\mathbb Z}^3} \varphi (2^{-j} |k|) \langle f, e_k \rangle e_k
\]
hold for $f\in {\mathcal D}'(\Lambda )$ and $j \in {\mathbb N}\cup \{ 0\}$.
We define the Besov norm $\| \cdot \| _{B_{p,r}^s}$ and the Besov space $B_{p,r}^s$ on $\Lambda$ with $s \in {\mathbb R}$ and $p,r \in [1,\infty]$ by
\begin{align*}
\| f \| _{B_{p,r}^s} &:= \left\{ \begin{array}{ll}
\displaystyle \left( \sum _{j=-1}^\infty 2^{jsr} \| \Delta _j f \| _{L^p}^r \right) ^{1/r} , & r\in [1,\infty) ,\\
\displaystyle \sup _{j \in {\mathbb N}\cup \{ -1,0\}} 2^{js} \| \Delta _j f \| _{L^p} , & r= \infty ,
\end{array} \right. \\
B_{p,r}^s& := \{ f\in {\mathcal D}'(\Lambda ); \| f \| _{B_{p,r}^s} <\infty \} .
\end{align*}
For simplicity of notation, we denote $B_{p,\infty }^s$ by $B_p^s$ for $s \in {\mathbb R}$ and $p \in [1,\infty]$.
Let
\[
S_j f := \sum _{k =-1}^{j-1} \Delta _{k} f, \quad j\in {\mathbb N}\cup \{ 0\} .
\]
For simplicity of notation, let $\Delta _{-2}f :=0$ and $S_{-1}f :=0$.
We define
\begin{align*}
f \mbox{\textcircled{\scriptsize$<$}} g &:= \sum _{j=0}^\infty (S_j f) \Delta _{j+1} g ,\quad f \mbox{\textcircled{\scriptsize$>$}} g := g \mbox{\textcircled{\scriptsize$<$}} f, \\
f \mbox{\textcircled{\scriptsize$=$}} g &:= \sum _{j=-1}^\infty \Delta _{j} f  \left( \Delta _{j-1} g + \Delta _{j} g + \Delta _{j+1} g\right) .
\end{align*}
By the definitions of $\{ \Delta _j\}$, $\{ S_j\}$, $\mbox{\textcircled{\scriptsize$<$}}$, $\mbox{\textcircled{\scriptsize$=$}}$, and $\mbox{\textcircled{\scriptsize$>$}}$, we have
\[
fg = f \mbox{\textcircled{\scriptsize$<$}} g + f \mbox{\textcircled{\scriptsize$=$}} g + f \mbox{\textcircled{\scriptsize$>$}} g .
\]
Let $f \mbox{\textcircled{\scriptsize$\leqslant$}} g := f \mbox{\textcircled{\scriptsize$<$}} g + f \mbox{\textcircled{\scriptsize$=$}} g$ and $f \mbox{\textcircled{\scriptsize$\geqslant$}} g := f \mbox{\textcircled{\scriptsize$>$}} g + f \mbox{\textcircled{\scriptsize$=$}} g$.
For the properties of Besov spaces and paraproducts, see Section 2 in \cite{AlKu} or \cite{BCD}.
We also remark that $P_N^{(1)}$ is a bounded operator on $B_p^s$ for $p\in (1,\infty )$ and $s\in {\mathbb R}$, and moreover, sufficiently good for commutator estimates with paraproducts (see Section 2 of \cite{AlKu}). 

Now we prepare notation of the polynomials of Ornstein-Uhlenbeck processes as follows.
\begin{align*}
{\mathcal Z}^{(1,N)}_t &:= P_N ^{(1)} Z_t ,\\
{\mathcal Z}^{(2,N)}_t &:= (P_N ^{(1)} Z_t) ^2 - C_1^{(N)} , \\
{\mathcal Z}^{(3,N)}_t &:= (P_N ^{(1)} Z_t) ^3 - 3 C_1^{(N)} P_N ^{(1)} Z_t , \\
{\mathcal Z}^{(0,2,N)}_t &:= \int _{-\infty}^t e^{(t-s)(\triangle -m_0^2)} P_N ^{(1)} {\mathcal Z}^{(2,N)}_s ds , \\
{\mathcal Z}^{(0,3,N)}_t &:= \int _{-\infty}^t e^{(t-s)(\triangle -m_0^2)} P_N ^{(1)} {\mathcal Z}^{(3,N)}_s ds , \\
{\mathcal Z}^{(2,2,N)}_t &:= {\mathcal Z}^{(2,N)}_t \mbox{\textcircled{\scriptsize$=$}}P_N ^{(1)} {\mathcal Z}^{(0,2,N)}_t -C_2^{(N)}, \\
{\mathcal Z}^{(2,3,N)}_t &:= {\mathcal Z}^{(2,N)}_t \mbox{\textcircled{\scriptsize$=$}}P_N ^{(1)} {\mathcal Z}^{(0,3,N)}_t - 3C_2^{(N)}{\mathcal Z}^{(1,N)}_t ,
\end{align*}
for $t\in (-\infty , \infty)$ and $N\in {\mathbb N}$.
Denote $P_N^{(2)}\tilde{X}^N$ by $X^{N}$. 
To show the tightness of the laws of  $\{ X^{N}\}$, by using these notations we transform (\ref{eq:SDEN3}) for a better equation.
In the present paper, we omit the detail of the transformation and just write the result of the transformation.
Consider the following:
\begin{align*}
X^{N,(2)}_t &:=  P_N^{(2)} \left( \tilde X_t^{N} -Z_t \right) + \lambda {\mathcal Z}^{(0,3,N)}_t \\
X^{N,(2),<}_t &:= -3\lambda \int _0^t e^{(t-s)(\triangle -m_0^2)} P_N^{(1)} \left[ \left( P_N^{(1)} X^{N,(2)}_s - \lambda P_N^{(1)} {\mathcal Z}^{(0,3,N)}_s \right) \mbox{\textcircled{\scriptsize$<$}} {\mathcal Z}^{(2,N)}_s \right] ds \\
X^{N,(2),\geqslant}_t & := X^{N,(2)}_t - X^{N,(2),<}_t .
\end{align*}
Note that $(X^{N,(2),<}_0, X^{N,(2),\geqslant}_0) = (0,X^{N,(2)}_0)= (0, P_N^{(2)} \left( \xi _N - \zeta \right) + \lambda {\mathcal Z}^{(0,3,N)}_0)$.
Let
\begin{align*}
\Psi _t^{(1)} (w) &:= \int _0^t e^{(t-s)(\triangle -m_0^2)} (P_N^{(1)} )^2  \left[ \left( w_s - \lambda P_N^{(1)} {\mathcal Z}^{(0,3,N)}_s \right) \mbox{\textcircled{\scriptsize$<$}} {\mathcal Z}^{(2,N)}_s \right] ds\\
&\qquad - \left( w_t - \lambda P_N^{(1)} {\mathcal Z}^{(0,3,N)}_t \right) \mbox{\textcircled{\scriptsize$<$}} \int _0^t e^{(t-s)(\triangle -m_0^2)} (P_N^{(1)} )^2 {\mathcal Z}^{(2,N)}_s ds , \\
\Psi _t^{(2)} (w) &:= \left[ \left( w_t - \lambda P_N^{(1)} {\mathcal Z}^{(0,3,N)}_t \right) \mbox{\textcircled{\scriptsize$<$}} \int _0^t e^{(t-s)(\triangle -m_0^2)} (P_N^{(1)} )^2 {\mathcal Z}^{(2,N)}_s ds \right] \mbox{\textcircled{\scriptsize$=$}} {\mathcal Z}^{(2,N)}_t\\
&\qquad - \left( w_t - \lambda P_N^{(1)} {\mathcal Z}^{(0,3,N)}_t \right) \left[ \int _0^t e^{(t-s)(\triangle -m_0^2)} (P_N^{(1)} )^2 {\mathcal Z}^{(2,N)}_s ds \mbox{\textcircled{\scriptsize$=$}} {\mathcal Z}^{(2,N)}_t \right] ,\\
\Phi _t^{(1)}(w)
&:= - 3 \left( {\mathcal Z}_t^{(1,N)} - \lambda P_N^{(1)} {\mathcal Z}^{(0,3,N)}_t\right) \mbox{\textcircled{\scriptsize$\leqslant$}} w_t^2 \\
&\qquad + 3 \lambda \left[ \left( 2{\mathcal Z}_t^{(1,N)} - \lambda P_N^{(1)} {\mathcal Z}^{(0,3,N)}_t \right) P_N^{(1)} {\mathcal Z}^{(0,3,N)}_t \right] \mbox{\textcircled{\scriptsize$\leqslant$}} w_t ,\\
\Phi _t ^{(2)} (w)
&:= - 3 \left( w_t - \lambda P_N^{(1)} {\mathcal Z}^{(0,3,N)}_t \right) \mbox{\textcircled{\scriptsize$>$}} {\mathcal Z}^{(2,N)}_t + 3 \lambda {\mathcal Z}^{(2,3,N)}_t \\
&\quad + 9\lambda \left( w_t - \lambda P_N^{(1)} {\mathcal Z}^{(0,3,N)}_t \right) \\
&\quad \hspace{2cm} \times \left( {\mathcal Z}^{(2,2,N)}_t - {\mathcal Z}^{(2,N)}_t\mbox{\textcircled{\scriptsize$=$}}\int _{-\infty }^0 e^{(t-s)(\triangle -m_0^2)} \left( P_N^{(1)}\right) ^2 {\mathcal Z}^{(2,N)}_s ds\right) \\
&\quad - \lambda ^2 \left( 3{\mathcal Z}_t^{(1,N)} - \lambda P_N^{(1)} {\mathcal Z}^{(0,3,N)}_t \right) \left( P_N^{(1)} {\mathcal Z}^{(0,3,N)}_t \right) ^2, \\
\Phi _t ^{(3)} (w)
&:= -3 \left( {\mathcal Z}_t^{(1,N)} - \lambda P_N^{(1)} {\mathcal Z}^{(0,3,N)}_t\right) \mbox{\textcircled{\scriptsize$>$}} w_t^2 \\
&\qquad + 3 \lambda \left[ \left( 2{\mathcal Z}_t^{(1,N)} - \lambda P_N^{(1)} {\mathcal Z}^{(0,3,N)}_t \right) P_N^{(1)} {\mathcal Z}^{(0,3,N)}_t \right] \mbox{\textcircled{\scriptsize$>$}} w_t .
\end{align*}
Then, in view of the argument in Section 4 of \cite{AlKu}, the pair $(X^{N,(2),<}_t, X^{N,(2),\geqslant}_t)$ satisfies the coupled partial differential equation:
\begin{equation} \label{PDEpara2}\left\{ \begin{array}{l}
\displaystyle (\partial _t - \triangle + m_0^2) X^{N,(2),<}_t \\[2mm]
\displaystyle = -3\lambda P_N^{(1)}  \left[ \left( P_N^{(1)} X^{N,(2),<}_t + P_N^{(1)} X^{N,(2),\geqslant}_t - \lambda P_N^{(1)} {\mathcal Z}^{(0,3,N)}_t \right) \mbox{\textcircled{\scriptsize$<$}} {\mathcal Z}^{(2,N)}_t \right] \\[4mm]
\displaystyle (\partial _t - \triangle + m_0^2) X^{N,(2),\geqslant}_t \\[1mm]
\displaystyle  = - \lambda P_N^{(1)} \left[ \left( P_N^{(1)} X^{N,(2),<}_t + P_N^{(1)} X^{N,(2),\geqslant}_t \right) ^3 \right] \\[2mm]
\displaystyle \quad + \lambda P_N^{(1)} \Phi _t^{(1)}(P_N^{(1)} X^{N,(2),<} + P_N^{(1)} X^{N,(2),\geqslant}) \\[2mm]
\displaystyle \quad + \lambda P_N^{(1)} \Phi _t^{(2)}(P_N^{(1)} X^{N,(2),<} + P_N^{(1)} X^{N,(2),\geqslant}) \\[2mm]
\displaystyle \quad + \lambda P_N^{(1)}  \Phi _t ^{(3)} (P_N^{(1)} X^{N,(2),<} + P_N^{(1)} X^{N,(2),\geqslant}) \\[2mm]
\displaystyle \quad - 3\lambda P_N^{(1)} \left[ ( P_N^{(1)} X^{N,(2),\geqslant}_t) \mbox{\textcircled{\scriptsize$=$}} {\mathcal Z}^{(2,N)}_t \right] \\[2mm]
\displaystyle \quad + 9\lambda ^2 P_N^{(1)} \left[ \Psi _t^{(1)} (P_N^{(1)} X^{N,(2),<} + P_N^{(1)} X^{N,(2),\geqslant} ) \mbox{\textcircled{\scriptsize$=$}} {\mathcal Z}^{(2,N)}_t \right] \\[2mm]
\displaystyle \quad + 9\lambda ^2 P_N^{(1)} \Psi _t^{(2)} (P_N^{(1)} X^{N,(2),<} + P_N^{(1)} X^{N,(2),\geqslant}) .
\end{array}\right. \end{equation}
By showing the tightness of the laws of $X^{N,(2)}_t = X^{N,(2),\geqslant}_t + X^{N,(2),<}_t$, we will obtain the tightness of the laws of $X_t^N := P_N^{(2)} \tilde X_t^N$.

\begin{rem}
Some typos in \cite{AlKu} are corrected in \eqref{PDEpara2}. Precisely, in \cite{AlKu}, $P_N^{(1)}$ of $P_N^{(1)} {\mathcal Z}^{(0,3,N)}_t$ is dropped in the equation corresponding to \eqref{PDEpara2} and also in the coefficients $\Psi ^{(i)}$ and $\Phi ^{(i)}$.  
\end{rem}

For estimates we prepare the following.
For $\eta \in [0,1)$, $\gamma \in (0,1/4)$ and $\varepsilon \in (0,1]$ define ${\mathfrak X}_{\lambda , \eta ,\gamma }^N (t)$ and  ${\mathfrak Y}_{\varepsilon}^N (t)$ by
\begin{align*}
{\mathfrak X}_{\lambda , \eta ,\gamma }^N (t) &:= \int _0^t \left( \left\| \nabla X_s^{N,(2),\geqslant}\right\| _{L^2}^2 + \left\| X_s^{N,(2)}\right\| _{L^2}^2 + \lambda \left\| P_N^{(1)} X_s^{N,(2)}\right\| _{L^4}^4 \right) ds \\
&\quad + \sup _{s',t' \in [0,t]; s'<t'} \frac{(s')^\eta \left\| X^{N,(2)}_{t'} - X^{N,(2)}_{s'} \right\| _{L^{4/3}}}{(t'-s')^{\gamma }} ,\\
{\mathfrak Y}_{\varepsilon}^N (t)&:= \int _0^t \left\| X^{N,(2),<}_s \right\| _{B_{4}^{1-\varepsilon }}^3 ds + \int _0^t \left\| X^{N,(2),\geqslant}_s \right\| _{B_{4/3}^{1+\varepsilon }} ds.
\end{align*}
To simplify the notation, we denote by $Q$ a positive polynomial built with the following quantities 
\begin{equation}\label{eq:Zs}\begin{array}{l}
\displaystyle \sup_{t\in [0,T]}\| {\mathcal Z}_t^{(1,N)}\| _{B_{\infty}^{-(1+\varepsilon )/2}}, \quad \sup_{t\in [0,T]}\| P_N^{(2)} Z_t \| _{B_{\infty}^{-(1+\varepsilon )/2}}, \quad \sup _{t\in [0,T]} \left\| {\mathcal Z}^{(2,N)}_t \right\| _{B_{\infty}^{-1-\varepsilon /24}},\\
\displaystyle \sup _{t\in [0,T]}  \left\| {\mathcal Z}^{(2,2,N)}_t \right\| _{B_{\infty }^{-\varepsilon /4}}, \quad \sup _{t\in [0,T]} \left\| {\mathcal Z}^{(0,2,N)}_t \right\| _{B_{\infty }^{1-\varepsilon /2}}, \quad \sup _{t\in [0,T]} \left\| {\mathcal Z}^{(0,3,N)}_t \right\| _{B_{\infty }^{1/2-\varepsilon /4}}, \\
\displaystyle \sup _{t\in [0,T]} \left\| {\mathcal Z}^{(2,3,N)}_t \right\| _{B_{\infty }^{-(1+\varepsilon )/2}}, \quad \sup _{t\in [0,T]} \left\| {\mathcal Z}_t^{(1,N)} \left( P_N^{(1)} {\mathcal Z}^{(0,3,N)}_t \right) \right\| _{B_{\infty }^{-(1+\varepsilon )/2}}, \\
\displaystyle \sup _{t\in [0,T]} \left\| {\mathcal Z}_t^{(1,N)} \left( P_N^{(1)} {\mathcal Z}^{(0,3,N)}_t \right) ^2\right\| _{B_{\infty }^{-(1+\varepsilon )/2}} \\
\displaystyle \hspace{5cm} \quad \mbox{and} \quad \sup _{s,t \in [0,T]; s<t} \frac{\left\| {\mathcal Z}^{(0,3,N)}_t - {\mathcal Z}^{(0,3,N)}_s \right\| _{L^\infty }}{(t-s)^{\gamma }},
\end{array}\end{equation}
with coefficients depending on $\lambda _0 $, $\varepsilon$, $\eta$, $\gamma$ and $T$, and we also denote by $C$ a positive constant depending on $\lambda _0$, $\varepsilon$, $\eta$, $\gamma$ and $T$.
We remark that $Q$ and $C$ can be different from line to line.
A constant depending on an extra parameter $\delta$ is denoted by $C_\delta$.
As in Section 3 of \cite{AlKu}, we have the square integrability of those in (\ref{eq:Zs}) with respect to the probability measure.
In view of this fact and hypercontractivity of Gaussian random variables, any polynomial consists of the elements in (\ref{eq:Zs}) are integrable with respect to the probability measure, i.e. $E[Q]\leq C$.

\section{Improvement of integrability}\label{sec:improve}

Let $\alpha \in [0,1/2)$ and choose $\varepsilon \in (0,1/16]$, $\gamma \in (0,1/8)$ and $\eta \in (1/2,1)$ such that $2\varepsilon < \gamma$, $\eta > \alpha + 2\gamma$ and $2\alpha + 4\gamma +\varepsilon < 1$.
In the present paper, we only see the difference from \cite{AlKu} and omit the argument of the parts which are the same as those in \cite{AlKu}.

We prepare some lemmas for estimates of the terms in \eqref{PDEpara2}, which are different versions of estimates in \cite{AlKu}.

\begin{lem}\label{lem:Psi1}
For $p\in [1,2]$, $\varepsilon \in (0,1/16)$, $s,t\in [0,T]$ and $\delta \in (0,1]$,
\begin{align*}
&\int _{s}^{t} (t-u)^{-\alpha /2 -\gamma }  \left\| \Psi _u^{(1)} (P_N^{(1)} X^{N,(2)}) \mbox{\rm\textcircled{\scriptsize$=$}} {\mathcal Z}^{(2,N)}_u \right\| _{B_{p}^{\varepsilon}} du\\
&\leq \delta \int _{s}^{t} \sup _{r\in [0,u)} \frac{r^\eta \left\| P_N^{(1)} X^{N,(2)}_u - P_N^{(1)} X^{N,(2)}_r \right\| _{L^{p}}}{(u-r)^{\gamma }} du \\
&\quad + \delta ^{-1} \int _0^t \left( \left\| X_u^{N,(2)} \right\| _{B_2^{15/16}}^2 + \left\| P_N^{(1)} X_u^{N,(2)}\right\| _{L^4}^4 \right) ^{7/8} du + \delta ^{-1} Q.
\end{align*}
\end{lem}

\begin{proof}
Choose $\theta \in (0,1/4)$.
Note that $\theta$ satisfies $\max\{ \eta , \alpha + \gamma + 2\theta + 3\varepsilon \} <3(1-\theta)/2$.
Applying Lemmas 4.3 and 2.3 in \cite{AlKu} and H\"older's inequality, we have
\begin{align*}
&\int _{s}^{t} (t-u)^{-\alpha /2 -\gamma } \left\| \Psi _u^{(1)} (P_N^{(1)} X^{N,(2)}) \mbox{\rm\textcircled{\scriptsize$=$}} {\mathcal Z}^{(2,N)}_u \right\| _{B_{p}^{\varepsilon}} du\\
&\leq Q \int _{s}^{t} (t-u)^{-\alpha /2 -\gamma } \left( \int _0^u (u-v)^{-21/32} \left\| P_N^{(1)} X_v^{N,(2)}\right\| _{B_{p}^{15/16}} dv \right) du\\
&\quad +Q \int _{s}^{t} (t-u)^{-\alpha /2 -\gamma } \left( \sup _{r\in [0,u)} \frac{r^\eta \left\| P_N^{(1)} X^{N,(2)}_u - P_N^{(1)} X^{N,(2)}_r \right\| _{L^{p}}}{(u-r)^{\gamma }} \right) ^{\theta}\\
&\quad \hspace{1cm} \times \left( \left\| P_N^{(1)} X_u^{N,(2)}\right\| _{L^p}^{1-\theta} + \int _0^u v^{-\eta /2} (u-v)^{(\gamma /2)-1-3\varepsilon /2} \left\| P_N^{(1)} X_v^{N,(2)}\right\| _{L^p}^{1-\theta} dv \right) du +Q\\
&\leq Q \int _0^t (t-u)^{-\alpha /2 -\gamma + 11/32}\left\| P_N^{(1)} X_v^{N,(2)}\right\| _{B_{p}^{15/16}} dv\\
&\quad + \delta \int _{s}^{t} \sup _{r\in [0,u)} \frac{r^\eta \left\| P_N^{(1)} X^{N,(2)}_u - P_N^{(1)} X^{N,(2)}_r \right\| _{L^{p}}}{(u-r)^{\gamma }} du\\
&\quad + \delta ^{-1} \int _{s}^{t} (t-u)^{-(\alpha +2\gamma)/[2(1-\theta)]} \\
&\quad \hspace{1cm} \times \left[ \left\| P_N^{(1)} X_u^{N,(2)}\right\| _{L^p} + \left( \int _0^u v^{-\eta /2} (u-v)^{(\gamma /2)-1-3\varepsilon /2} \left\| P_N^{(1)} X_v^{N,(2)}\right\| _{L^p}^{1-\theta} dv \right) ^{1/(1-\theta)} \right] du \\
&\quad +Q\\
&\leq Q \int _0^t\left\| P_N^{(1)} X_v^{N,(2)}\right\| _{B_{p}^{15/16}}^{7/4} dv + \delta \int _{s}^{t} \sup _{r\in [0,u)} \frac{r^\eta \left\| P_N^{(1)} X^{N,(2)}_u - P_N^{(1)} X^{N,(2)}_r \right\| _{L^{p}}}{(u-r)^{\gamma }} du\\
&\quad + \delta ^{-1} C \left( \int _s^t (t-u)^{-(\alpha +2\gamma )/(1-\theta )} du\right) \left( \int _{s}^{t}\left\| P_N^{(1)} X_u^{N,(2)}\right\| _{L^p}^2 du \right) ^{1/2} +Q \\
&\quad + \delta ^{-1} C \int _{s}^{t} (t-u)^{-(\alpha +2\gamma)/[2(1-\theta)]} u^\theta \\
&\quad \hspace{4cm} \times \left( \int _0^u v^{-\eta /[2(1-\theta )]} (u-v)^{(\gamma -2-3\varepsilon )/[2(1-\theta)]} \left\| P_N^{(1)} X_v^{N,(2)}\right\| _{L^p} dv \right) du .
\end{align*}
Noting that H\"older's inequality and Lemma 2.3 in \cite{AlKu} imply
\begin{align*}
&\int _{s}^{t} (t-u)^{-(\alpha +2\gamma)/[2(1-\theta)]} u^\theta \left( \int _0^u v^{-\eta /[2(1-\theta )]} (u-v)^{(\gamma -2-3\varepsilon )/[2(1-\theta)]} \left\| P_N^{(1)} X_v^{N,(2)}\right\| _{L^p} dv \right) du\\
&\leq C \int _0^t (t-u)^{-(\alpha +2\gamma)/[2(1-\theta)]} \left( \int _0^u v^{-\eta /[2(1-\theta )]} (u-v)^{(\gamma -2-3\varepsilon )/[2(1-\theta)]} \left\| P_N^{(1)} X_v^{N,(2)}\right\| _{L^p} dv \right) du\\
&\leq C \int _0^t v^{-\eta /[2(1-\theta )]} (t-v)^{-(\alpha + \gamma + 2\theta + 3\varepsilon )/[2(1-\theta)]\}} \left\| P_N^{(1)} X_v^{N,(2)}\right\| _{L^p} dv\\
&\leq C \left( \int _0^t v^{-2\eta /[3(1-\theta )]} (t-v)^{-2(\alpha + \gamma + 2\theta + 3\varepsilon )/[3(1-\theta)]} dv\right) ^{3/4} \left( \int _0^t \left\| P_N^{(1)} X_v^{N,(2)}\right\| _{L^p}^4 dv\right) ^{1/4},
\end{align*}
we obtain the assertion in view of the choice of $\theta$ and H\"older's inequality again.
\end{proof}

\begin{lem}\label{lem:Phi3}
For $p\in [1,2]$, $\varepsilon \in (0,1/16)$, $t\in [0,T]$ and $\delta \in (0,1]$,
\begin{align*}
\left\| \Phi _t ^{(3)} (P_N^{(1)} X^{N,(2)}) \right\| _{B_{p}^{-(1+\varepsilon)/2}} &\leq \delta \left\| P_N^{(1)} X^{N,(2)}_t\right\| _{L^4}^4 + \delta ^{-1} Q .
\end{align*}
\end{lem}

\begin{proof}
Estimates of the paraproducts (see Proposition 2.1 (ii) in \cite{AlKu}) imply
\begin{align*}
\left\| \Phi _t ^{(3)} (P_N^{(1)} X^{N,(2)}) \right\| _{B_{p}^{-(1+\varepsilon)/2}} &\leq Q\left( \left\| \left(P_N^{(1)} X_t^{N,(2)} \right) ^2 \right\| _{L^p} + \left\| P_N^{(1)} X_t^{N,(2)}\right\| _{L^p}\right) \\
&\leq Q \left\| P_N^{(1)} X_t^{N,(2)}\right\| _{L^4}^2 + Q \\
&\leq \delta \left\| P_N^{(1)} X_t^{N,(2)} \right\| _{L^4}^4 + \delta ^{-1}Q.
\end{align*}
Thus, we have the inequality.
\end{proof}

\begin{lem}\label{lem:=Z2}
For $\varepsilon \in (0,1/16)$, $p\in [1,2]$, $t\in [0,T]$ and $\delta \in (0,1]$,
\begin{align*}
&\left\| (P_N^{(1)} X^{N,(2),\geqslant}_t) \mbox{\rm\textcircled{\scriptsize$=$}} {\mathcal Z}^{(2,N)}_t \right\| _{B_{p}^{\varepsilon /8}} \\
&\leq \delta \left( \left\| \nabla X^{N,(2),\geqslant}_t \right\| _{L^2}^2 + \left\| P_N^{(1)} X^{N,(2)}_t \right\| _{L^4} ^4 \right) + \delta \left\| X^{N,(2),<}_t \right\| _{L^p}^2 + \delta \left\| X^{N,(2),\geqslant}_t \right\| _{B_{p}^{1+\varepsilon }} + \delta ^{-2}Q.
\end{align*}
\end{lem}

\begin{proof}
An estimates of the resonance term (see Proposition 2.1 (iv) in \cite{AlKu}) implies
\begin{equation}\label{eq:lem=Z2-1}
\left\| (P_N^{(1)} X^{N,(2),\geqslant}_t) \mbox{\textcircled{\scriptsize$=$}} {\mathcal Z}^{(2,N)}_t \right\| _{B_{p}^{\varepsilon /8}} \leq C\left\| {\mathcal Z}^{(2,N)}_t \right\| _{B_{\infty}^{-1-\varepsilon /8}} \left\| P_N^{(1)} X^{N,(2),\geqslant}_t \right\| _{B_{p}^{1+\varepsilon /4}} .
\end{equation}
By the interpolation inequality of Besov spaces (see Proposition 2.1 (vii) in \cite{AlKu}) we have
\begin{equation}\label{eq:lem=Z2-2}\begin{array}{rl}
\displaystyle \left\| P_N^{(1)} X^{N,(2),\geqslant}_t \right\| _{B_{p}^{1+\varepsilon /4}}
&\displaystyle \leq \left\| P_N^{(1)} X^{N,(2),\geqslant}_t \right\| _{B_{p}^{1-\varepsilon /8}}^{2/3} \left\| P_N^{(1)} X^{N,(2),\geqslant}_t \right\| _{B_{p}^{1+\varepsilon }}^{1/3} \\
&\displaystyle \leq \delta \left\| P_N^{(1)}X^{N,(2),\geqslant}_t \right\| _{B_{p}^{1-\varepsilon /8}}^2 + C \delta ^{-1/2} \left\| P_N^{(1)} X^{N,(2),\geqslant}_t \right\| _{B_{p}^{1+\varepsilon }}^{1/2} .
\end{array}\end{equation}
In view of
\begin{align*}
&\left\| P_N^{(1)} X^{N,(2),\geqslant}_t \right\| _{B_{p}^{1-\varepsilon /8}}^2
\leq C \left\| P_N^{(1)} X^{N,(2),\geqslant}_t \right\| _{W^{1-\varepsilon /8,p}}^2\\
&\leq C \left(  \left\| P_N^{(1)} X^{N,(2)}_t \right\| _{L^p} + \left\| \nabla P_N^{(1)} X^{N,(2),\geqslant}_t \right\| _{L^p} \right) ^2 + C \left\| P_N^{(1)} X^{N,(2),<}_t \right\| _{L^p}^2,
\end{align*}
from \eqref{eq:lem=Z2-1} and \eqref{eq:lem=Z2-2} we have
\begin{align*}
&\left\| (P_N^{(1)} X^{N,(2),\geqslant}_t) \mbox{\textcircled{\scriptsize$=$}} {\mathcal Z}^{(2,N)}_t \right\| _{B_{p}^{\varepsilon /8}} \\
&\leq \delta Q \left[ \left(  \left\| P_N^{(1)} X^{N,(2)}_t \right\| _{L^p} + \left\| \nabla P_N^{(1)} X^{N,(2),\geqslant}_t \right\| _{L^p} \right) ^2 + \left\| P_N^{(1)} X^{N,(2),<}_t \right\| _{L^p}^2 \right] \\
&\quad + \delta ^{-1/2} Q \left\| P_N^{(1)} X^{N,(2),\geqslant}_t \right\| _{B_{p}^{1+\varepsilon }}^{1/2} + Q, \\
&\leq \delta Q \left(  \left\| P_N^{(1)} X^{N,(2)}_t \right\| _{L^p}^4 + \left\| \nabla P_N^{(1)} X^{N,(2),\geqslant}_t \right\| _{L^p}^2 \right) + \delta Q \left\| P_N^{(1)} X^{N,(2),<}_t \right\| _{L^p}^2 \\
&\quad + \delta \left\| P_N^{(1)} X^{N,(2),\geqslant}_t \right\| _{B_{p}^{1+\varepsilon }} + \delta ^{-2} Q.
\end{align*}
Hence, by replacing $\delta$ and using the uniform boundedness of $P_N^{(1)}$ in $N$ (see Proposition 2.5 in \cite{AlKu}) we obtain the assertion.
\end{proof}

The following proposition is an improved version of Proposition 4.13 in \cite{AlKu}, and actually the regularity of the Besov space is improved by $\alpha$.

\begin{prop}\label{prop:holder+}
For $t\in [0,T]$, 
\begin{align*}
&E\left[ \sup _{s',t'\in [0,t]; s'<t'} \frac{(s')^\eta \left\| X^{N,(2)}_{t'} - X^{N,(2)}_{s'} \right\| _{B_{4/3}^\alpha}}{(t'-s')^{\gamma }} \right] \\
&\leq C E\left[ \sup _{r\in [0,t]} r^\eta \left\| X^{N,(2),\geqslant}_{r} \right\| _{B_{4/3}^{\alpha + 2\gamma}}  \right]  + CE\left[ \sup _{r\in [0,t]} r^\eta \left\| X^{N,(2),<}_r \right\| _{B_{4/3}^{\alpha + 2\gamma}} \right] \\
&\quad + C E\left[ \left\| X_t^{N,(2),<} \right\| _{L^2}^2 \right] + C E\left[ {\mathfrak Y}_{\varepsilon}^N (t) \right] + C\sup _{s\in [0,t]} E\left[ \left\| X_s^{N,(2)} \right\| _{B_1^{-1/2+\varepsilon}} ^q\right] + C.
\end{align*}
\end{prop}

\begin{proof}
In view of (\ref{PDEpara2}) it follows that
\begin{align*}
&X^{N,(2),<}_t - e^{(t-s)(\triangle - m_0^2)} X^{N,(2),<}_s\\
& = -3 \lambda \int _s^t e^{(t-u)(\triangle - m_0^2)} P_N^{(1)} \left[ \left( P_N^{(1)} X^{N,(2)}_u - \lambda P_N^{(1)} {\mathcal Z}^{(0,3,N)}_u \right) \mbox{\textcircled{\scriptsize$<$}} {\mathcal Z}^{(2,N)}_u \right] du
\end{align*}
for $s,t \in [0,T]$ such that $s<t$.
Hence, for $s',t' \in [0,T]$ such that $s'<t'$, the smoothing property of the heat semigroup and an estimate of the paraproduct (see Proposition 2.1 in \cite{AlKu}) imply
\begin{align*}
&\left\| X^{N,(2),<}_{t'} - X^{N,(2),<}_{s'} \right\| _{B_{4/3}^\alpha} \\
&\leq \left\| e^{(t'-s')(\triangle - m_0^2)} -I \right\| _{B_{4/3}^{\alpha + 2\gamma} \rightarrow B_{4/3}^\alpha} \left\| X^{N,(2),<}_{s'} \right\| _{B_{4/3}^{\alpha + 2\gamma}} \\
&\quad + 3 \lambda \int _{s'}^{t'} \left\| e^{(t'-u)(\triangle - m_0^2)} P_N^{(1)} \left[ \left( P_N^{(1)} X^{N,(2)}_u - \lambda P_N^{(1)} {\mathcal Z}^{(0,3,N)}_u \right) \mbox{\textcircled{\scriptsize$<$}} {\mathcal Z}^{(2,N)}_u \right] \right\| _{B_{4/3}^\alpha}du \\
&\leq C (t'-s')^{\gamma } \left\| X^{N,(2),<}_{s'} \right\| _{B_{4/3}^{\alpha + 2\gamma}} \\
&\quad + C \lambda \int _{s'}^{t'} (t'-u)^{-(\alpha +1)/2-\varepsilon /2} \left\| \left( P_N^{(1)} X^{N,(2)}_u - \lambda P_N^{(1)} {\mathcal Z}^{(0,3,N)}_u \right) \mbox{\textcircled{\scriptsize$<$}} {\mathcal Z}^{(2,N)}_u \right\| _{B_{4/3}^{-1-\varepsilon}}du \\
&\leq C(t'-s')^{\gamma } \left\| X^{N,(2),<}_{s'} \right\| _{B_{4/3}^{\alpha + 2\gamma}} \\
&\quad + \lambda Q (t'-s')^{\gamma } \int _{s'}^{t'} (t'-u)^{-(\alpha +2\gamma +1 + \varepsilon)/2} \left\| P_N^{(1)} X^{N,(2)}_u - \lambda P_N^{(1)} {\mathcal Z}^{(0,3,N)}_u \right\| _{L^{4/3}}du .
\end{align*}
Thus, by applying H\"older's inequality we have for $t \in [0,T]$ and $\delta \in (0,1]$
\begin{equation}\label{eq:propglobal2-3-01}
\begin{array}{l}
\displaystyle \sup _{s',t' \in [0,t]; s'<t'} \frac{(s')^\eta \left\| X^{N,(2),<}_{t'} - X^{N,(2),<}_{s'} \right\| _{B_{4/3}^\alpha}}{(t'-s')^{\gamma }} \\
\displaystyle \leq C \sup _{r\in [0,t]} \left( r^\eta \left\| X^{N,(2),<}_r \right\| _{B_{4/3}^{\alpha + 2\gamma}} \right) + \delta \lambda \int _0^t \left\| P_N^{(1)} X^{N,(2)}_u \right\| _{L^{4/3}}^4 du + C_\delta Q.
\end{array}
\end{equation}
Similarly, from (\ref{PDEpara2}), for $s',t' \in [0,T]$ such that $s'<t'$, we have the estimate
\begin{align*}
&\left\| X^{N,(2),\geqslant}_{t'} - X^{N,(2),\geqslant}_{s'} \right\| _{B_{4/3}^\alpha} \\
&\leq C(t'-s')^{\gamma } \left\| X^{N,(2),\geqslant}_{s'} \right\| _{B_{4/3}^{\alpha + 2\gamma}} + C \lambda (t'-s')^{\gamma } \int _{s'}^{t'} (t'-u)^{-\alpha /2 -\gamma } \left\| P_N^{(1)} X^{N,(2)}_u \right\| _{L^4}^3 du \\
&\quad + C \lambda (t'-s')^{\gamma } \int _{s'}^{t'} (t'-u)^{-\alpha /2 -\gamma } \left\| \Phi _u^{(1)}(P_N^{(1)} X^{N,(2)} ) \right\| _{L^{4/3}} du \\
&\quad + C \lambda (t'-s')^{\gamma } \int _{s'}^{t'} (t'-u)^{-\alpha /2 -\gamma -1/4 -\varepsilon /2} \left\| \Phi _u ^{(2)} (P_N^{(1)} X^{N,(2)}) \right\| _{B_{4/3}^{-1/2-\varepsilon}} du \\
&\quad + C \lambda (t'-s')^{\gamma } \int _{s'}^{t'} (t'-u)^{-\alpha /2 -\gamma -1/4 -\varepsilon /2} \left\| \Phi _u ^{(3)} (P_N^{(1)} X^{N,(2)}) \right\| _{B_{4/3}^{-1/2-\varepsilon}} du \\
&\quad + C \lambda (t'-s')^{\gamma } \int _{s'}^{t'} (t'-u)^{-\alpha /2 -\gamma }  \left\| (P_N^{(1)} X^{N,(2),\geqslant}) \mbox{\textcircled{\scriptsize$=$}} {\mathcal Z}^{(2,N)}_u \right\| _{L^{4/3}} du \\
&\quad + C \lambda (t'-s')^{\gamma } \int _{s'}^{t'} (t'-u)^{-\alpha /2 -\gamma }  \left\| \Psi _u^{(1)} (P_N^{(1)} X^{N,(2)}) \mbox{\textcircled{\scriptsize$=$}} {\mathcal Z}^{(2,N)}_u \right\| _{B_{4/3}^{\varepsilon}} du \\
&\quad + C \lambda (t'-s')^{\gamma } \int _{s'}^{t'} (t'-u)^{-\alpha /2 -\gamma }  \left\| \Psi _u^{(2)} (P_N^{(1)} X^{N,(2)}) \right\| _{B_{4/3}^{\varepsilon}} du .
\end{align*}
For $\delta \in (0,1]$, applying Lemmas 4.4, 4.5 and 4.7 in \cite{AlKu} and Lemmas \ref{lem:Phi3} and \ref{lem:=Z2} with replacing $\delta$ by $(t'-u) ^\beta$ with suitable $\beta$ for each lemmas, and applying Lemma \ref{lem:Psi1} and H\"older's inequality, we have for $\delta \in (0,1]$ $s',t' \in [0,T]$ such that $s'<t'$
\begin{align*}
&\left\| X^{N,(2),\geqslant}_{t'} - X^{N,(2),\geqslant}_{s'} \right\| _{B_{4/3}^\alpha} \\
&\leq C(t'-s')^{\gamma } \left\| X^{N,(2),\geqslant}_{s'} \right\| _{B_{4/3}^{\alpha + 2\gamma}} + C \lambda (t'-s')^{\gamma } \int _{s'}^{t'} \left( \left\| \nabla X_u^{N,(2),\geqslant}\right\| _{L^2}^2 + \left\| P_N^{(1)} X_u^{N,(2)}\right\| _{L^4}^4 \right) du\\
&\quad + C \lambda (t'-s')^{\gamma }\int _0^t \left( \left\| X_u^{N,(2)} \right\| _{B_2^{15/16}}^2 + \left\| P_N^{(1)} X_u^{N,(2)}\right\| _{L^4}^4 \right) ^{7/8} du \\
&\quad + C (t'-s')^{\gamma } \int _{s'}^{t'} \left\| X^{N,(2),<}_u \right\| _{L^{4/3}}^2 du + Q (t'-s') ^{\gamma } \int _{s'}^{t'} \left\| X^{N,(2),\geqslant}_u \right\| _{B_{4/3}^{1+\varepsilon }} du\\
&\quad + \delta Q (t'-s')^{\gamma } \int _{s'}^{t'} \sup _{r\in [0,u)} \frac{r^\eta \left\| P_N^{(1)} X^{N,(2)}_u - P_N^{(1)} X^{N,(2)}_r \right\| _{L^{4/3}}}{(u-r)^{\gamma }} du + C_\delta Q (t'-s')^{\gamma } .
\end{align*}
Here, we remark that applying Lemmas \ref{lem:Phi3} and \ref{lem:=Z2} instead of Lemmas 4.8 and 4.9 in \cite{AlKu} respectively, enables us to improve the regularity of the estimate by $\alpha \in [0,1/2)$.
It is also remarked that Lemma \ref{lem:Psi1} is provided for the clarity of the proof.

From this inequality and (\ref{eq:propglobal2-3-01}) we obtain the conclusion by following the proof of Proposition 4.12 in \cite{AlKu}.
\end{proof}

The following proposition is an improved version of Proposition 4.17 in \cite{AlKu}, and again the regularity of the Besov space is improved by $\alpha$.
We need the version, because the supremum in time of the norms on $B_{4}^{\alpha + 2\gamma }$ and $B_{4/3}^{\alpha + 2\gamma }$ appeared in Proposition \ref{prop:holder+}.

\begin{prop}\label{prop:estsup+}
For $q\in (1, 8/7) $, $t\in [0,T]$ and $\delta \in (0,1]$, we have
\begin{align*}
&E\left[ \sup _{r\in [0,t]} r^\eta \left\| X^{N,(2),<}_{r} \right\| _{B_{4}^{\alpha + 2\gamma }}^3 \right] + E\left[ \sup _{r\in [0,t]} r^\eta \left\| X^{N,(2),\geqslant}_{r} \right\| _{B_{4/3}^{\alpha + 2\gamma }} \right]\\
&\leq C E\left[ \left\| X^{N,(2)}_0 \right\| _{B_{4/3}^{\alpha + 2\gamma - 2\eta }} \right] + C \delta E\left[{\mathfrak X}_{\lambda , \eta ,\gamma}^N (t) \right] + C \delta E\left[ {\mathfrak Y}_{\varepsilon}^N (t) ^q \right] + C_\delta .
\end{align*}
\end{prop}

\begin{proof}
By Lemma 4.14(i) in \cite{AlKu} we have
\begin{align*}
&E\left[ \sup _{r\in [0,t]} r^\eta \left\| X^{N,(2),<}_{r} \right\| _{B_{4}^{\alpha + 2\gamma }}^3 \right] \\
&\leq \lambda E\left[ Q \sup _{r\in [0,t]} \left( \int _0^r (r-u)^{-(1+\alpha)/2 -\gamma -\varepsilon /4} \left\| P_N^{(1)} X^{N,(2)}_u - \lambda P_N^{(1)} {\mathcal Z}^{(0,3,N)}_u \right\| _{L^{4}} du \right) ^3\right] .
\end{align*}
Hence, by applying H\"older's inequality we have for $\delta \in (0,1]$
\begin{equation}\label{eq:propestsup01}
E\left[ \sup _{r\in [0,t]} r^\eta \left\| X^{N,(2),<}_{r} \right\| _{B_{4}^{\alpha + 2\gamma }}^3 \right] \leq \delta \lambda E\left[ \int _0^t  \left\| P_N^{(1)} X^{N,(2)}_u\right\| _{L^4}^4 du \right] + C_\delta.
\end{equation}
Similarly to the proof of Lemma 4,14(ii) in \cite{AlKu} we have for $s,t\in [0,T]$ such that $s<t$
\begin{align*}
\left\| X_t^{N,(2),\geqslant} \right\| _{B_{4/3}^{\alpha + 2\gamma}}
&\leq C (t-s)^{-\eta} \left\| X_s^{N,(2),\geqslant} \right\| _{B_{4/3}^{\alpha + 2\gamma -2\eta}} \\
&\quad + C\lambda \int _s^t (t-u)^{-(\alpha +2\gamma)/2} \left\| P_N^{(1)} X_u^{N,(2)} \right\| _{L^4}^3 du \\
&\quad + C\lambda \int _s^t (t-u)^{-(\alpha +2\gamma)/2} \left\| \Phi _u^{(1)}(P_N^{(1)} X^{N,(2)}) \right\| _{L^{4/3}} du \\
&\quad + C\lambda \int _s^t (t-u)^{-(2\alpha +4\gamma +1+2\varepsilon)/4} \left\| \Phi _u^{(2)}(P_N^{(1)} X^{N,(2)})\right\| _{B_{4/3}^{-1/2-\varepsilon}} du \\
&\quad + C\lambda \int _s^t (t-u)^{-(2\alpha +4\gamma +1+2\varepsilon)/4} \left\| \Phi _u^{(3)}(P_N^{(1)} X^{N,(2)})\right\| _{B_{4/3}^{-1/2-\varepsilon}} du \\
&\quad + C\lambda \int _s^t (t-u)^{-(\alpha +2\gamma)/2} \left\| ( P_N^{(1)} X_u^{N,(2),\geqslant}) \mbox{\textcircled{\scriptsize$=$}} {\mathcal Z}_u^{2,N} \right\| _{L^{4/3}} du \\
&\quad + C\lambda \int _s^t (t-u)^{-(\alpha +2\gamma -\varepsilon )/2} \left\|\Psi _u^{(1)}(P_N^{(1)} X^{N,(2)}) \mbox{\textcircled{\scriptsize$=$}} {\mathcal Z}_u^{2,N} \right\| _{B_{4/3}^{\varepsilon}} du \\
&\quad + C\lambda \int _s^t (t-u)^{-(\alpha +2\gamma -\varepsilon )/2} \left\|\Psi _u^{(2)}(P_N^{(1)} X^{N,(2)}) \right\| _{B_{4/3}^{\varepsilon}} du .
\end{align*}
Similarly to the proof of Proposition \ref{prop:holder+}, for $\delta \in (0,1]$, applying Lemmas 4.4, 4.5 and 4.7 in \cite{AlKu} and Lemmas \ref{lem:Phi3} and \ref{lem:=Z2} with replacing $\delta$ by $\delta (t-u) ^\beta$ with suitable $\beta$ for each lemmas, and applying Lemma 3.1 and H\"older's inequality, we have
\begin{align*}
&E\left[ \sup _{r\in [0,t]} r^\eta \left\| X^{N,(2),\geqslant}_{r} \right\| _{B_{4/3}^{\alpha + 2\gamma }} \right] \\
&\leq C E\left[ \left\| X^{N,(2)}_0 \right\| _{B_{4/3}^{\alpha + 2\gamma  - 2\eta }} \right] + \delta E\left[{\mathfrak X}_{\lambda ,\eta ,\gamma}^N (t) \right] + \delta E\left[ {\mathfrak Y}_{\varepsilon}^N (t) ^q \right] + C_\delta.
\end{align*}
Here, we use the assumptions of the parameters $\alpha$, $\gamma$ and $\varepsilon$.
Therefore, by this inequality and (\ref{eq:propestsup01}) we have the assertion.
\end{proof}

Now we obtain the following uniform estimate in $N$.

\begin{thm}\label{thm:tight1}
Let $\alpha \in [0,1/2)$ and choose $\varepsilon \in (0,1/16]$, $\gamma \in (0,1/8)$ and $\eta \in (1/2,1)$ such that $2\varepsilon < \gamma$, $\eta > \alpha + 2\gamma$ and $2\alpha + 4\gamma +\varepsilon < 1$, and let $q\in (1, 8/7)$.
Then, we have
\begin{align*}
&E\left[ \sup _{s,t\in [0,T]; s<t} \frac{s^\eta \left\| X^{N,(2)}_{t} - X^{N,(2)}_{s} \right\| _{B_{4/3}^\alpha}}{(t-s)^{\gamma }} \right] + E\left[{\mathfrak X}_{\lambda , \eta ,\gamma}^N (T) \right] + E\left[ {\mathfrak Y}_{\varepsilon}^N (T) ^q \right] \\
&+ E\left[ \sup _{r\in [0,T]} r^\eta \left\| X^{N,(2),<}_{r} \right\| _{B_{4}^{\alpha + 2\gamma }}^3 \right] + E\left[ \sup _{r\in [0,T]} r^\eta \left\| X^{N,(2),\geqslant}_{r} \right\| _{B_{4/3}^{\alpha + 2\gamma }} \right]\\
& \leq C.
\end{align*}
\end{thm}

\begin{proof}
By following the proof of Theorem 4.18 in \cite{AlKu} with applying Propositions \ref{prop:holder+} and \ref{prop:estsup+} instead of Propositions 4.13 and 4.17 in \cite{AlKu}, we obtain the assertion.
\end{proof}

Theorem \ref{thm:tight1} improves the regularity of Besov norms in Theorem 4.18 in \cite{AlKu} by $\alpha$.
By using the improvement we are able to show the tightness of the laws of $\{ X^N\}$ in the spaces smaller than that in Theorem 4.19 in \cite{AlKu} as follows.

\begin{thm}\label{thm:tight2}
For $\tilde \varepsilon \in (0, 1/16]$, the laws of $\{ X^N\}$ are tight on $C([0,\infty ); {B_{12/5}^{-1/2- \tilde \varepsilon }})$.
Moreover, if $X$ is a limit in law of a subsequence $\{ X^{N(k)}\}$ of $\{ X^N\}$ on $C([0,\infty ); {B_{12/5}^{-1/2-\tilde \varepsilon }})$, then $X$ is a continuous process on $B_{12/5}^{-1/2-\tilde \varepsilon }$, the limit measure $\mu$ of the associated subsequence $\{ \mu _{N(k)}\}$ is a stationary measure with respect to $X$ and it holds that
\begin{equation}\label{eq:thmtight2}
\int \| \phi \| _{B_\infty ^{-1/2-\tilde \varepsilon}}^2 \mu (d\phi ) < \infty .
\end{equation}
\end{thm}

\begin{proof}
We follow the proof of Theorem 4.19 in \cite{AlKu}.
Choose $\alpha \in [0,1/2)$ sufficiently close to $1/2$ so that $\alpha + \tilde \varepsilon >1/2$, choose $\gamma \in (0,1/8)$ and $\varepsilon \in (0,1/16)$ sufficiently small, and choose $\eta \in (1/2,1)$ sufficiently large so that the assumptions in Theorem \ref{thm:tight1} and $\varepsilon < \tilde \varepsilon$ hold.
Let $T\in (0,\infty )$ and $t_0 \in (0,T)$.
For $h\in (0,1]$ and $\varepsilon' \in (0,1]$, Chebyshev's inequality implies that
\begin{align*}
&\sup _{N\in {\mathbb N}} P \left( \sup _{s,t\in [t_0,T]; |s-t|<h} \left\| X^{N,(2)}_t - X^{N,(2)}_s \right\| _{B_{4/3}^\alpha} > \varepsilon' \right) \\
&\leq \frac{h^\gamma }{\varepsilon' t_0^\eta } E\left[ \sup _{s,t\in [t_0,T]; s<t, t-s<h} \frac{s^\eta \left\| X^{N,(2)}_{t} - X^{N,(2)}_{s} \right\| _{B_{4/3}^\alpha}}{(t-s)^{\gamma }} \right] 
\end{align*}
Hence, from Theorem \ref{thm:tight1} we obtain
\begin{equation}\label{eq:thmtight2-1}
\lim _{h\downarrow 0} \sup _{N\in {\mathbb N}} P \left( \sup _{s,t\in [t_0,T]; |s-t|<h} \left\| X^{N,(2)}_t - X^{N,(2)}_s \right\| _{B_{4/3}^\alpha} > \varepsilon' \right) =0
\end{equation}
for $\varepsilon' \in (0,1]$.
On the other hand, Chebyshev's inequality implies that, for any $R>0$,
\[
\sup _{N\in {\mathbb N}} P \left( \left\| X^{N,(2)}_{t_0} \right\|  _{B_{4/3}^{\alpha + 2\gamma }} >R \right) 
\leq \frac{1}{R t_0^\eta} \sup _{N\in {\mathbb N}} E\left[ \sup _{r\in [0,T]} r^\eta \left\| X^{N,(2)}_{r} \right\| _{B_{4/3}^{\alpha + 2\gamma }}\right].
\]
Hence, by Theorem \ref{thm:tight1} we obtain
\begin{equation}\label{eq:thmtight2-2}
\lim _{R\rightarrow \infty} \sup _{N\in {\mathbb N}} P \left( \left\| X^{N,(2)}_{t_0} \right\|  _{B_{4/3}^{\alpha + 2\gamma }} >R \right) =0 .
\end{equation}
In view of the fact that the unit ball in $B_{4/3}^{\alpha + 2\gamma }$ is compactly embedded in $B_{4/3}^{\alpha}$ (see Theorem 2.94 in \cite{BCD}), the tightness of the laws of $\{ X^{N,(2)}\}$ on $C([t_0,T]; {B_{4/3}^{\alpha}})$ follows from (\ref{eq:thmtight2-1}) and (\ref{eq:thmtight2-2}).
By the Besov embedding theorem (see Proposition 2.1 in \cite{AlKu}) we have $B_{4/3}^{\alpha} \subset B_{12/5}^{-1/2-\tilde \varepsilon}$.
Hence, we have the tightness of the laws of $\{ X^{N,(2)}\}$ on $C([t_0,T]; B_{12/5}^{-1/2-\tilde \varepsilon})$.
The rest of the proofs are completely same as that of Theorem 4.19 in \cite{AlKu} except (\ref{eq:thmtight2}).

Now we prove (\ref{eq:thmtight2}).
The stationarity of $X^N$ implies
\begin{align*}
\int \| \phi \| _{B_\infty ^{-1/2-\tilde \varepsilon}}^2 \mu (d\phi ) &\leq \liminf _{k\rightarrow \infty} \int \| \phi \| _{B_\infty ^{-1/2-\tilde \varepsilon}}^2 \mu _{N(k)}(d\phi ) \\
&= \liminf _{k\rightarrow \infty} E\left[ \left\| X^{N(k)}_0 \right\| _{B_\infty ^{-1/2-\tilde \varepsilon}}^2 \right] \\
&= \frac{1}{T} \liminf _{k\rightarrow \infty} \int _0^T E\left[ \left\| X^{N(k)}_t \right\| _{B_\infty ^{-1/2-\tilde \varepsilon}}^2 \right] dt \\
&\leq C \liminf _{k\rightarrow \infty} \int _0^T E\left[ \left\| X^{N(k), (2)}_t \right\| _{B_\infty ^{-1/2-\tilde \varepsilon}}^2 \right] dt + C.
\end{align*}
Since the Besov embedding theorem implies
\[
\left\| X^{N(k), (2)}_t \right\| _{B_\infty ^{-1/2-\tilde \varepsilon}}^2 \leq C \left\| X^{N(k), (2)}_t \right\| _{B_2^{1-\tilde \varepsilon}}^2,
\]
we have
\begin{align*}
\int \| \phi \| _{B_\infty ^{-1/2-\varepsilon}}^2 \mu (d\phi ) 
&\leq C \liminf _{N\rightarrow \infty} \int _0^T E\left[ \left\| X^{N, (2)}_t \right\| _{B_2^{1-\tilde \varepsilon}}^2 \right] dt + C\\
&\leq C \liminf _{N\rightarrow \infty} E\left[{\mathfrak X}_{\lambda , \eta ,\gamma}^N (T) \right] + C \liminf _{N\rightarrow \infty} E\left[ {\mathfrak Y}_{\varepsilon}^N (T) \right] +C.
\end{align*}
Therefore, we obtain (\ref{eq:thmtight2}) from Theorem \ref{thm:tight1}.
\end{proof}

\vspace{5mm}
\noindent
{\bf Acknowledgements.}
The author thanks the anonymous referees for helpful comments. The comments improved the quality of the present paper.
This work was partially supported by JSPS KAKENHI Grant Numbers 17K14204 and 21H00988.


\def\cprime{$'$} \def\cprime{$'$} \def\cprime{$'$}


\begin{thebibliography}{10}

\bibitem{AlKu}
S.~Albeverio and Sei.~Kusuoka, 
\newblock The invariant measure and the flow associated to the $\Phi ^4_3$-quantum field model,
\newblock {\em Ann. Sc. Norm. Super. Pisa Cl. Sci. (5)} {\bf{20}} (2020), 1359--1427. \verb#https://doi.org/10.2422/2036-2145.201809_008#

\bibitem{BCD}
H.~Bahouri, J.-Y. Chemin and R.~Danchin,
\newblock {\em Fourier analysis and nonlinear partial differential equations},
  volume 343 of {\em Grundlehren der Mathematischen Wissenschaften [Fundamental Principles of Mathematical Sciences]}.
\newblock Springer, Heidelberg, 2011. \verb#https://doi.org/10.1007/978-3-642-16830-7#

\bibitem{BaGu1}
N.~Barashkov and M.~Gubinelli,
\newblock A variational method for $\Phi ^{4}_{3}$,
\newblock {\em Duke Math. J.} {\bf{169}} (2020), 3339--3415. \verb#https://doi.org/10.1215/00127094-2020-0029#

\bibitem{BaGu2}
N.~Barashkov and M.~Gubinelli,
\newblock The $\Phi^4_3$ measure via Girsanov's theorem,
\newblock {\em Electron. J. Probab.} {\bf{26}} (2021), Paper No. 81, 29pp. \verb#https://doi.org/10.1214/21-EJP635#

\bibitem{CaCh}
R.~Catellier and K.~Chouk,
\newblock Paracontrolled distributions and the 3-dimensional stochastic quantization equation,
\newblock {\em Ann. Probab.} {\bf{46}} (2018), no. 5, 2621–2679. \verb#https://doi.org/10.1214/17-AOP1235#

\bibitem{GuHo1}
M.~Gubinelli and M.~Hofmanova,
\newblock Global solutions to elliptic and parabolic $\Phi ^4$ models in Euclidean space,
\newblock {\em Comm. Math. Phys.} {\bf{368}} (2019), no. 3, 1201–1266. \verb#https://doi.org/10.1007/s00220-019-03398-4#

\bibitem{GuHo2}
M.~Gubinelli and M.~Hofmanova,
\newblock A PDE construction of the Euclidean $\Phi ^4_3$ quantum field theory,
\newblock to appear in {\em Comm. Math. Phys.}, arXiv:1810.01700. \verb#https://doi.org/10.1007/s00220-021-04022-0#

\bibitem{GIP}
M.~Gubinelli, P~Imkeller and N.~Perkowski,
\newblock Paracontrolled distributions and singular PDEs,
\newblock {\em Forum Math. Pi} {\bf{3}} (2015), e6, 75 pp. \verb#https://doi.org/10.1017/fmp.2015.2#

\bibitem{Ha}
M.~Hairer,
\newblock A theory of regularity structures,
\newblock {\em Invent. Math.} {\bf{198}} (2014), no. 2, 269–504. \verb#https://doi.org/10.1007/s00222-014-0505-4#

\bibitem{Ho}
M.~Hoshino,
\newblock Global well-posedness of complex Ginzburg-Landau equation with a space-time white noise,
\newblock {\em Ann. Inst. Henri Poincaré Probab. Stat.} {\bf{54}} (2018), no. 4, 1969–2001. \verb#https://doi.org/10.1214/17-AIHP862#

\bibitem{HIN}
M.~Hoshino, Y.~Inahama and N.~Naganuma,
\newblock Stochastic complex Ginzburg-Landau equation with space-time white noise, 
\newblock {\em Electron. J. Probab.} {\bf{22}} (2017), Paper No. 104, 68 pp. \verb#https://doi.org/10.1214/17-EJP125#


\bibitem{MoWe}
A.~Moinat and H.~Weber.
\newblock Space-time localisation for the dynamic model.
\newblock {\em Communications on Pure and Applied Mathematics},
  {\bf{73}} (2020), 2519--2555. \verb#https://doi.org/10.1002/cpa.21925#

\bibitem{MW3}
J.-C.~Mourrat and H.~Weber,
\newblock The dynamic ${\Phi}_3^4$ model comes down from infinity,
\newblock {\em Commun. Math. Phys.}, {\bf{356}}(3):673--753, 2017. \verb#https://doi.org/10.1007/s00220-017-2997-4#

\bibitem{ZhZh1}
R.~Zhu and X.~Zhu,
\newblock Lattice approximation to the dynamical ${\Phi}_3^4$ model,
\newblock {\em Ann. Probab.} {\bf{46}} (2018), no. 1, 397–455. \verb#https://doi.org/10.1214/17-AOP1188#

\bibitem{ZhZh2}
R.~Zhu and X.~Zhu,
\newblock Dirichlet form associated with the ${\Phi}_3^4$ model,
\newblock {\em Electron. J. Probab.} {\bf{23}} (2018), Paper No. 78, 31 pp. \verb#https://doi.org/10.1214/18-EJP207#

\end{thebibliography}
\end{document}